\documentclass[11pt]{amsart}
\usepackage[english]{babel}
\usepackage[usenames,dvipsnames]{pstricks}
\usepackage{epsfig}
\usepackage{graphicx}
\usepackage{amssymb}
\usepackage[all]{xy}

\newtheorem{theorem}{Theorem}[section]
\newtheorem{proposition}[theorem]{Proposition}
\newtheorem{corollary}[theorem]{Corollary}

\newcommand{\cs}{\mathcal{C}}               % configuration space
\newcommand{\pcs}{\mathrm{P}(\mathcal{C})}  % paths in configuration space
\newcommand{\ws}{\mathcal{W}}               % working space
\newcommand{\qs}{\mathcal{Q}}               % queries
\newcommand{\as}{\mathcal{A}}

\newcommand{\Ker}{\mathord{\mathrm{Ker}}}

\newcommand{\Id}{\mathord{\mathrm{Id}}}
\newcommand{\gsec}{\mathord{\mathrm{sec}}}
\newcommand{\csec}{\mathord{\mathrm{csec}}}
\newcommand{\nil}{\mathord{\mathrm{nil}}}

\newcommand{\RR}{\mathord{\mathbb{R}}}
\newcommand{\ZZ}{\mathord{\mathbb{Z}}}

\newcommand{\cx}{\mathord{\mathrm{cx}}}  % complexity 
\newcommand{\TC}{\mathord{\mathrm{TC}}}
\newcommand{\cat}{\mathord{\mathrm{cat}}}

\title{COMPLEXITY OF THE FORWARD KINEMATIC MAP}
\author[Petar Pave\v{s}i\'c]{Petar Pave\v{s}i\'c}
\address{Faculty of Mathematics and Physics, University of Ljubljana, Ljubljana, Slovenia}
\email{\rm{petar.pavesic@fmf.uni-lj.si}}
\begin{document}
\begin{abstract}
The main objective of this paper is to introduce a new method for  qualitative analysis of various designs of robot arms.
To this end we define the complexity of a map, examine its main properties and develop some methods of computation.
In particular, when applied to a forward kinematic map associated to some robot arm structure, the complexity measures 
the inherent discontinuities that arise when one attempts to solve the motion planning problem for any set of input data. 
In the second part of the paper, we consider instabilities of motion planning in the proximity of singular points,  
and present explicit computations for several common robot arm configurations.
\end{abstract}
\maketitle

\section{Introduction}

In this paper we introduce and discuss a new qualitative measure of the complexity of a forward kinematic map from 
the configuration space
of the robot arm joints to the working space of the end-effector. 
Let us illustrate the problem on a familiar example of a robot arm
with  $n$ revolute joints. The position of the $i$-th joint is uniquely determined by some angle of rotation $\theta_i$, 
so we may identify the position
of each joint with a point on the unit circle $T$, and the combined position of all $n$ joints with an $n$-tuple of values 
$(\theta_1,\ldots,\theta_n)\in T^n=T\times\ldots\times T$ ($n$ factors). The position of the end-effector is determined 
by the spatial
location and the orientation, so it corresponds to a point $(\vec{r},R)$ in the cartesian product $\RR^3\times  SO(3)$ 
(here we identify the space of all possible orientation of a rigid body  with the set $SO(3)$ of all orthogonal $3\times 3$ matrices of determinant 1). The exact form of the resulting
 forward kinematic
map $F\colon T^n\to\RR^3\times SO(3)$ depends on the lengths and the respective inclinations of the axes, and is usually
given in terms of Denavit-Hartenberg matrices but the explicit formulae will not be relevant for the main part of our 
discussion.

The motion planning problem in this setting may be stated as follows: given an initial state of joint parameters $(\theta_1,\ldots,\theta_n)\in T^n$ and 
a required end-effector position $(\vec{r},R)\in\RR^3\times SO(3)$, find motions of the joints starting at $(\theta_1,\ldots,\theta_n)$ 
and ending in a position of joints such that the corresponding  end-effector position is $(\vec{r},R)$. 
The problem may be modelled as follows: let $P(T^n)$ denote the space of all possible paths in $T^n$ (i.e. continuous maps $\alpha\colon [0,1]\to T^n$ from the 
interval $[0,1]$ to the joint parameter space), and let $\pi\colon P(T^n)\to T^n\times (\RR^3\times SO(3))$ be the map that to each path assigns its starting
and ending position, $\pi(\alpha):=\big(\alpha(0),F(\alpha(1))\big)$. 
Then the solution of the motion planning problem can be viewed as an inverse map $\rho \colon T^n\times (\RR^3\times SO(3))\to P(T^n)$, with the property 
$\pi(\rho((\theta_1,\ldots,\theta_n),(\vec{r},R)))=((\theta_1,\ldots,\theta_n),(\vec{r},R))$. We are mainly interested 
in robust motion plans, such that a small perturbation of the initial 
data results in a comparatively small perturbation of the corresponding motion plan. 
In other words, we normally require that the map $\rho$ is continuous with respect to the input 
data.

The starting point of our investigation is the following fundamental observation (see Theorem \ref{thm: cx 1 implies section}):
\begin{center}
\begin{minipage}{120mm}
\emph{If there exists a robust global solution of the motion planning
problem for the map $F$, then there also exists a continuous global solution of the inverse kinematic problem for the map $F$.}
\end{minipage}  
\end{center}
Since in practice inverse kinematic solutions can be found 
only for a very restricted class of simple manipulators, it follows that a solution of the motion planning problem will almost always require 
a partition of the input data space into smaller domains, over which a robust motion plan can be constructed. The minimal number of 
domains that is needed to cover all possible input data measures 
the complexity of the motion planning for a given robot arm configuration. We are going to describe a mathematical model that will allow 
a clear definition of the complexity and develop several methods for its computation. 

\subsection{Prior work}

Motion planning is one of the basic problems in robotics, it has been extensively studied under all possible aspects, and there exists a vast literature
on the topic. We are going to rely on \cite{Hollerbach 89}  and \cite{Kavraki-LaValle} as basic references. In the standard formulation of the motion planning problem 
one specifies the  \emph{configuration space} $\cs$ of the robot device (which normally correspond to the set of joint parameters but may also take into account various 
limitations), the \emph{working space} $\ws$ (spatial position and orientation that can be reached by the robot), and  \emph{obstacle regions} in $\ws$. 
Then one considers \emph{queries} that consist of an \emph{initial configuration} $q_I\in\cs$ and a \emph{goal configuration} $q_G\in\cs$
(corresponding to the desired position and orientation of the end-effector), and asks
for a path $\rho\colon [0,1]\to\cs$ that avoids obstacles, and for which $\rho(0)=q_I$ and $\rho(1)=q_G$. 

The complexity of motion planning was mostly considered within the context of computational complexity: indeed, 
the search for explicit algorithms aimed to the solution of a given motion planning problem was often accompanied by more
general considerations regarding the algorithmic complexity of various solutions - 
see \cite{Reif 79} and \cite{Canny 88}. 
For a more recent study that extensively uses topological methods and is similar in spirit to our approach see
 \cite{Erdmann 10}.

A more geometrical measure for the complexity of motion planning was introduced by M. Farber \cite{Farber:TCMP} who observed that algorithmic 
solutions normally yield robust mapping plans. He then defined the concept of the \emph{topological complexity} of motion planning in the working space 
of a mechanical device as the minimal number of continuous partial solutions to the motion planning problem. In many cases the computation of 
the topological complexity of a space may be reduced to the computation of a very classical numerical invariant, called the Lusternik-Schnirelmann category. 
Indeed, in certain sense the study of the complexity of the geometric motion planning can be traced back to the 1930's, to the work in variational calculus by 
L. Lusternik and L. Schnirelmann. They introduced what is today called the \emph{Lusternik-Shnirelmann category} of a space, denoted $\cat(X)$, as a tool 
to estimate the number of critical points of a smooth map. Their work was widely extended both in analysis, most notably by J. Schwartz \cite{Schwartz} and 
R. Palais \cite{Palais}, and in topology,  by R. Fox \cite{Fox}, T. Ganea \cite{Ganea}, I. James \cite{James} and many others.  Today Lusternik-Schnirelmann 
category is a well-developed theory with many ramifications and methods of computation techniques that allow to systematically determine the category for most of the spaces 
that will appear in this paper - see \cite{CLOT}. It is interesting to see how this very classical and independently developed theory found its application
in the study of problems in robotics. For an overview of principal results on topological complexity see \cite{Farber:ITR}, for a less technical and very
readable account, see also \cite[Chapter 8]{Ghrist 14}.

The study of the complexity of a map is a natural continuation of Farber's work and was suggested as a problem by A. Dranishnikov during the conference on Applied 
Algebraic Topology in Castro Urdiales (Spain, 2014). In spite of strong formal similarities, the flavour of  this new theory is quite different from 
the topological complexity. A partial explanation can be found in some earlier work by J. Hollerbach \cite{Hollerbach} and D. Gottlieb \cite{Gottlieb:IEEE}, who studied 
the possibility to avoid singularities of  the forward kinematic map by introducing additional joints. They proved that under standard assumptions a forward kinematic 
map always has singularities and (with rare exceptions) does not admit  global inverse kinematics. As a consequence, the study of the complexity of a map seems to
be less amenable to purely homotopy-theoretical methods.

\subsection{Our contribution}

We introduce a general framework for the study of the complexity of a continuous map. Roughly speaking, 
the complexity of a map $F\colon\cs\to\ws$ is the minimal number of robust rules that take as input pairs of points $(c,w)\in\cs\times\ws$, and yield
paths $\rho=\rho(c,w)$ starting at $c$ and ending at some $c'$ that is mapped by $F$ to $w$. We first show that under the assumption that $F$ is regular and 
admits a right inverse (section), the computation of the complexity of $F$ can be reduced to the topological complexities of the spaces $\cs$ and $\ws$ in the sense of Farber.
However, a forward kinematic map usually satisfy these assumptions only locally, which causes considerable difficulties and require the development of entirely
new methods. As a main computational tool we introduce the concept of relative complexity of $F$ with respect to  suitably chosen subspaces of $\cs\times\ws$.  
In specific situations we normally proceed in two stages, by first finding a suitable decomposition of $\cs\times\ws$ and then using different 
estimates to determine the complexity of $F$ over each piece. In this way we are able to obtain good estimates for some  important practical cases, 
especially various combinations of revolute joints.

\subsection{Outline}

In the next section  we define the complexity of a forward kinematic map and compare it with some related concepts.
The central part of the paper are Sections 3 and 4. Section 3 contains the theoretical background and is divided into four subsections, where we consider
kinematic maps that admit inverses, regular kinematic maps, lower bounds related to some cohomological obstructions, and finally the general situation, 
where the kinematic map is not assumed to be regular or possess a global inverse. In Section 4 we apply the theory to estimate the complexity of several
important examples, including the system of two parallel joints, the universal joint, the triple-roll wrist and the 6-DOF joint. 
We suggest to the reader to read Section 4 in parallel with Section 3 in order to appreciate the significance of various theoretical results.

Throughout the paper we use the following standard notation: $T$ is the unit circle in the plane, $T^n$ is the cartesian product of $n$ circles,
$S^2$ is the two-dimensional sphere, and $SO(3)$ is the space of rotations of the three-dimensional space, or equivalently, of orthogonal 
$3\times 3$-matrices with determinant 1. Moreover ${\rm pr}_\cs$ denotes projection of a product to the factor $\cs$ and ${\rm gr}(F)$ denotes the graph
of the map $F$.

\section{Complexity of a map}

As mentioned in the Introduction, the basic example that we have in mind is the forward kinematic map of a robot arm with $n$ revolute joints. 
However, in order to allow other types 
of joints and arm configurations, we are going to work in a more general setting. Let us therefore consider an arbitrary forward kinematic map 
$F\colon \cs\to \ws$ from some \emph{configuration space} $\cs$ to a \emph{working space} $\ws$. Normally the space $\cs$ is the cartesian product of 
the parameter spaces for the individual joints (but the space may also be restricted to reflect  various obstacles and other conditions), 
while $\ws$ is a subspace of $\RR^3\times SO(3)$ (but may be enlarged to include the velocity and angular momentum of the end-effector). 
A motion of the arm is then simply a path in $\cs$, which we model by a continuous map $\rho\colon [0,1]\to \cs$.  
All theoretically possible motions of the arm are described by the set of all paths in the joint parameter space, which we denote by $\pcs$. 
The setting of time-scale as an interval between 0 and 1 allows a simple description of the initial and final stage of the motion, so we have a map
$\pi\colon \pcs\to \cs\times \ws$, given by $\pi(\rho)=(\rho(0),F(\rho(1))$, i.e. to each movement 
of the arm we assign its initial position of the joints
and the goal position of the end-effector. The pair $(c,w)\in\cs\times \ws$ representing an initial configuration $c$ and a goal position $w$ 
is often called a \emph{query}, and $\cs\times \ws$ is the \emph{query space}. A map $\rho\colon\cs\times \ws\to \pcs$ for which $\pi(\rho(c,w))=(c,w)$ is called 
a \emph{roadmap}, because it may be interpreted as a rule that to each initial configuration $c$ and goal position $w$ of the end-effector assigns a movement of joints
$\rho(c,w)$ that starts in the configuration $c$ and ends in the position $w$. More formally, a roadmap is a right inverse $\rho$ to the projection $\pi$. 
We require that the map $\rho$ is robust, which means that a small perturbation of the initial
data results in a small change of the path performed by the robot arm. In more mathematical terms, this amounts to the requirement that the map $\rho$ must be continuous.

We are now ready to state our first conclusion that may be viewed as the starting point of this study. 

\begin{theorem}
\label{thm: cx 1 implies section}
If a given forward kinematic map $F\colon \cs\to \ws$ admits a robust roadmap, then it also admits an continuous inverse kinematic map.
\end{theorem}
\begin{proof}
Let $\rho\colon \cs\times \ws\to \pcs$ be a robust roadmap for $F$, and let $c_0$ be some initial configuration of the robot arm. 
Then the formula $I(w):=\rho(c_0,w)(1)$ satisfies the relation $F(I(w))=w$ for every $w\in\ws$, which means that $I\colon \ws\to \cs$ is the inverse kinematic map for $F$.
\end{proof}

We are, of course, mostly interested in the negative aspect of the above result: since most forward kinematic maps that appear in practice do not admit a continuous
inverse kinematic map, it follows that most motion planning problems cannot have a global robust solution. We are therefore forced to look for robust 
solutions on subspaces of the query space $\cs\times \ws$. The minimal number of domains with robust roadmaps that cover all possible queries in
$\cs\times\ws$ will be called the \emph{complexity of the forward kinematic map} $F$ and will  be  denoted $\cx(F)$.

Observe that if we assume that $\cs$ and $\ws$ are the same space and that $F$ is  the identity map on $\cs$, then  $\cx(\Id_\cs)$ coincides with $\TC(\cs)$, 
the topological complexity of $\cs$ 
that we mentioned in the Introduction. Indeed, that case corresponds to motion planning within the configuration space of the robot arm, without taking into account
the relation to the working space, given by the forward kinematic map. The topological complexity of most of the spaces that are of interest for us has 
already been computed and can be found in the literature so we will  systematically reduce the computation of $\cx(F)$ to the computation of the topological
complexity of $\cs$ and $\ws$. Moreover, we will simplify the notation and write $\cx(\cs)$ and $\cx(\ws)$ instead of $\cx(\Id_\cs)$ and $\cx(\Id_\ws)$. 

Similarly as in the case of the topological complexity, we will not try to compute the complexity of $F$ directly but rather by 
finding suitable upper and lower estimates. Upper estimates are mostly based on explicit 
partitioning of the query space and description of corresponding robust roadmaps. Lower estimates are more subtle, as they require to theoretically demonstrate 
the impossibility to find a smaller number of robust roadmaps.

\section{Estimates of $\cx(F)$}

Let us introduce a more formal definition of the complexity that will ensure mathematical correctness of our conclusions. 
A space is said to be an  \emph{Euclidean Neighbourhood Retract} (short: ENR) if it can be obtained as a retract of some open subspace of $\RR^n$. 
This class includes most interesting geometric objects like manifolds, polyhedra, algebraic sets and other spaces that arise as configuration and working spaces 
of mechanical systems. 
We will assume that $\cs$ and $\ws$ are ENR spaces and will consider only zzs of $\cs$ and $\ws$ that are also ENR. 
A \emph{partial roadmap} for the forward kinematic map $F\colon\cs\to\ws$ is a continuous map $\rho\colon \qs\to \pcs$ whose domain $\qs$ is an ENR subspace
of the query space  $\cs\times \ws$, and $\rho$ is a right inverse for $\pi$, i.e. $\pi(\rho(c,w))=(c,w)$ for every query $(c,w)\in\qs$.
The \emph{complexity} of $F$ is the minimal integer $n$ for which $\cs\times\ws$ can be covered by $n$ partial roadmaps.  
We will usually take the domains of roadmaps to be disjoint but that requirement is not
part of the definition, and sometimes it may even be more natural to allow overlapping (for example, 
when we want to take into account various inaccuracies and noise that arise in real-world situations). 

\subsection{Invertible kinematic maps}

In this subsection we will work under the assumption that $F\colon \cs\to \ws$ admits a continuous inverse kinematic map $I\colon \ws\to\cs$, such that 
$F(I(w))=w$ for all $w\in \ws$. This assumption is rarely satisfied in practice, 
but we will be still able to apply the results when the inverse kinematic map
can be computed over parts of $\ws$ (i.e. avoiding singular points and gimbal lock positions). 
We have the following basic result that gives good upper and lower estimates for the complexity of $F$.

\begin{theorem}
\label{thm: sectioned F}
If the map $F\colon \cs\to \ws$ admits a right inverse $I\colon \ws\to \cs$ then 
$$\cx(\ws)\le \cx(F) \le \cx(\cs).$$
\end{theorem}
\begin{proof}
To prove that $\cx(\ws)\le \cx(F)$ we will show that a partition of $\cs\times\ws$ into $n$ partial roadmaps for $F$ allows to generate a partition of 
$\ws\times\ws$ into $n$ partial roadmaps for $\Id_\ws$. For every partial roadmap $\rho\colon \qs\to \pcs$  the following formula 
$$\bar \rho(w,w'):=F\circ \rho(I(w),w')$$
clearly determines a partial roadmap on $\overline \qs=\{(w,w')\mid (I(w),w')\in \qs\}$. Moreover, if the domains  $\qs_1,\ldots,\qs_n$ cover 
$\cs\times \ws$ then the corresponding domains $\overline \qs_1,\ldots,\overline \qs_n$ cover $\ws\times\ws$. 

Similarly, given a subspace $\qs\subseteq \cs\times\cs$ and a partial roadmap $\rho\colon\qs\to\pcs$ for $\Id_\cs$, the formula 
$$\bar \rho(c,w):=\rho(c,I(w))$$
determines a partial roadmap on 
$$\overline \qs=\{(c,w)\mid (c,I(w))\in \qs\}.$$ 
If the domains $\qs_1,\ldots,\qs_n$ cover 
$\cs\times \cs$ then the corresponding domains $\overline \qs_1,\ldots,\overline \qs_n$ cover $\cs\times\ws$, therefore $\cx(F)\le\cx(\cs)$.
\end{proof}

The configuration space of a $n$-joint robot arm is the cartesian product of $n$-circles, $\cs=T^n$, whose topological complexity is 
exactly $n+1$ (see \cite[Theorem 13]{Farber:TCMP}), so if there exists a global robust inverse kinematic map for $F$, then $\cx(F)\le n+1$. 
On the other side, the complexity of the standard working space $\ws=\RR^3\times SO(3)$ is 4 (see \cite[Theorem 4.61]{Farber:ITR}), so in the presence 
of a global inverse kinematic map the motion planning requires at least 4 robust partial roadmaps. This surprising result explains many of the practical
difficulties that arise when we try to construct explicit motion planning algorithms. We will see later, that for systems that do not admit a global inverse kinematics
the minimal number of robust roadmaps may be even bigger.

We have seen in Theorem \ref{thm: cx 1 implies section} that if the complexity of a forward kinematic map $F$ is 1  then $F$ admits a global
inverse kinematic map, and so by Theorem \ref{thm: sectioned F} the complexity of the working space is also 1. By \cite[Theorem 1]{Farber:TCMP} 
spaces with complexity one are contractible (i.e continuously deformable to a point). This is possible only if the robot working space is a spatial domain
without obstacles, and its movement does not involve planar or spatial rotations (cf. \cite[Section 5.6.1]{Handbook}).

\subsection{Regular kinematic maps}

In this subsection we take into account the analytic properties of the forward kinematic map $F\colon\cs\to\ws$. In fact, $F$ is normally a smooth mapping
so we may consider its Jacobian $J(F)$ which can at every point be represented by an $m\times n$ matrix, where $m$ and $n$ are respectively the dimensions
of the spaces $\ws$ and $\cs$. In particular the dimension of the configuration space corresponds to the \emph{degrees of freedom} of the robot system.
If at some point $c\in\cs$ the rank of the Jacobian matrix of $F$ is not maximal, then from that point the device cannot move in all possible directions
in the working space. This is a common problem in robotic systems, known as a \emph{gimbal lock}. The forward kinematic map $F$ is \emph{regular} at $c\in\cs$ 
if the Jacobian matrix of $F$ at that point has maximal rank, otherwise $F$ is \emph{singular} at $c$. A forward kinematic map $F$ is \emph{regular} if 
it is regular at all points. Hollerbach \cite{Hollerbach}  proved that a forward kinematic map whose working space allows arbitrary rotations of the end-effector 
always has singular points, even if the system is redundant  (cf. also Gottlieb \cite{Gottlieb}). Nevertheless, in this section we are going to consider the 
complexity of regular forward kinematic maps as an intermediate step toward the general case. 

For our purposes, the main property of regular kinematic maps is that they allow to lift paths from the working space to the configuration space in 
the following sense. Let $c$ be a configuration in $\cs$, and let $\alpha$ be a path in $\ws$ (corresponding to a sequence of movements of the end-effector) 
that starts at $\alpha(0)=F(c)$. We may interpret the pair $(c,\alpha)$ as an input datum for the following task: find a sequence of motions of the joints 
that starts from the joint configuration $c$, and such that the end-effector performs exactly the movements prescribed by the path $\alpha$. It is a standard fact
of differential topology due to Ehresmann \cite{Ehresmann} that if $F$ is regular, then this task has a robust solution. 
More formally, let us denote by $\cs\sqcap P(\ws)=\{(c,\alpha)\in\cs\times P(\ws)\mid F(c)=\alpha(0)\}$  the space of input tasks for a robot system.
Then Ehresmann's theorem may be stated as follows.

\begin{proposition}
If the forward kinematic map $F\colon\cs\to\ws$ is regular, then there exists a continuous map $\Gamma\colon\cs\sqcap P(\ws)\to\pcs$ such that 
for all $(c,\alpha)\in \cs\sqcap P(\ws)$ we have $\Gamma(c,\alpha)(0)=c$ and $F\circ\Gamma(c,\alpha)=\alpha$.
\end{proposition}

The property stated in the proposition essentially means that we may solve every motion task in the configuration space, provided we are able to solve 
the corresponding task in the working space. Thus the following result does not come as a surprise.

\begin{theorem}
\label{thm: F fibration}
If the forward kinematic map $F\colon\cs\to\ws$ is regular then $$\cx(F)\le\cx(\ws).$$
\end{theorem}
\begin{proof}
Let $\qs\subseteq \ws\times\ws$ and let $\rho\colon\qs\to P(\ws)$ be a partial roadmap for $\Id_\ws$. Then the formula 
$$\bar \rho(c,w):=\Gamma(c,\rho(F(c),w))$$
determines a partial roadmap on $\overline \qs=\{(c,w)\mid (F(c),w)\in \qs\}$. If the domains $\qs_1,\ldots,\qs_n$ cover 
$\ws\times \ws$ then the corresponding domains $\overline \qs_1,\ldots,\overline \qs_n$ cover $\cs\times\ws$, therefore $\cx(F)\le\cx(\ws)$.
\end{proof}

By combining Theorems \ref{thm: sectioned F} and \ref{thm: F fibration} we obtain the following corollary:

\begin{corollary}
\label{cor: sectioned fibration}
If the forward kinematic map is regular and admits an inverse kinematic map, then $\cx(F)=\cx(\ws)$.
\end{corollary}

In the next result we give a precise characterization of a query set that can admit a robust roadmap.

\begin{proposition}
\label{prop: cx=1}
Let $F\colon \cs\to\ws$ be a regular forward kinematic map. Then a set of queries $\qs\subset\cs\times\ws$ admits a robust partial roadmap $\rho \colon\qs\to\pcs$ 
if, and only if $\qs$ can be deformed (within $\cs\times\ws$) to the graph of $F$. 
\end{proposition}
\begin{proof}
Given a robust roadmap $\rho\colon\qs\to\pcs$ we can define a deformation $D\colon\qs\times [0,1]\to \cs\times\ws$ by the formula
$$D(c,w,t):=(\rho(c,w)(t),w).$$
Clearly, the initial stage of deformation is $D(c,w,0)=(\rho(c,w)(0),w)=(c,w)$, and the final stage is $D(c,w,1)=(\rho(c,w)(1),w)=(c',w)$,
where $F(c')=w$, therefore $D(\qs\times\{1\})$ is contained in the graph of $F$.

Conversely, if $D\colon\qs\times[0,1]\to\cs\times\ws$ is a deformation of $\qs$ to the graph of $F$, then the projections of $D(c,w,t)$  to $\cs$ and $\ws$ yield paths
$\alpha$ from $c$ to $c'$ in $\cs$ and $\alpha'$ from $w$ to $w'$ in $\ws$, such that $F(c')=w'$. Therefore, we may join the path $\alpha$ with the reverse of the lifting 
of $\alpha'$ to obtain a motion plan from $c$ to $w$. The corresponding formula for the roadmap $\rho\colon\qs\to\pcs$ is thus
$$\rho(c,w)(t):=\left\{\begin{array}{ll}
{\rm pr}_\cs(D(c,w,2t)) &;\ \ 0\le t\le\frac{1}{2}\\
\Gamma \big(({\rm pr}_\cs(D(c,w,1)),{\rm pr}_\ws(D(c,w,-))\big)(2-2t)&;\ \ \frac{1}{2}\le t\le 1\end{array}\right.
$$
\end{proof}

Observe that the regularity of $F$ was used only in the second half of the proof. In fact, a roadmap on $\qs$ always determines a deformation of $\qs$ to the graph of $F$,
and moreover, during the deformation the $\ws$-component is preserved. We will say that the roadmap defines a \emph{horizontal} deformation of $\qs$ to the graph.

\subsection{Cohomological lower bound}
\label{subsec: cohomology}

The upper bounds for the complexity of $F$ that we obtained in the last two subsections are actually constructive, 
being derived from the complexities of $\cs$ and $\ws$ for which suitable roadmaps can be explicitly described. 
On the other side, the lower bound in Theorem \ref{thm: sectioned F} depends indirectly on the lower bound for $\cx(\ws)$,
which is in turn based on some cohomological estimates as in \cite[Section 4.5]{Farber:ITR}. In this subsection we will 
obtain better estimates by considering the homomorphism in cohomology induced by the forward kinematic map. 

Let $H^*$ be any cohomology theory (e.g. de Rham, singular, \v Cech...) and assume that a set of queries $\qs$ admits 
a roadmap $\rho\colon\qs\to\pcs$. Then by the above discussion $\qs$ may be continuously deformed to the graph of $F$ which may be expressed by the following diagram
$$\xymatrix{ 
{\cs} \ar[rr]^{\Id_\cs\times F} & & {\cs\times\ws}\\
& {\qs} \ar[lu]^{d} \ar@{^(->}[ru]_i} 
$$
where $d(c,w)=\rho(c,w)(1)$, the image of  $\Id_\cs\times F$ is exactly the graph of $F$, and $(\Id_\cs\times F)\circ d$ 
is homotopic to the inclusion $i$. 
By applying the contravariant cohomology functor $H^*$ we obtain a commutative diagram of respective cohomology groups
$$\xymatrix{ 
H^*({\cs}) \ar[rd]_{d^*} & & {H^*(\cs\times\ws)}\ar[ll]_{\Id\times F^*}\ar[ld]^{i^*}\\
& {H^*(\qs)}  } 
$$
If a cohomology class $u$ is in the kernel of $\Id\times F^*$ then clearly $i^*(u)=0$. The kernel of $i^*$ coincides 
with the image of the homomorphism 
$$j^* \colon H^*(\cs\times\ws,\qs)\to H^*(\cs\times\ws),$$ 
therefore one can find a relative cohomology class $\bar u\in H^*(\cs\times\ws,\qs)$ such that $u=j^*(\overline u)$.
In particular, if $\cx(F)=1$ then every element of $\Ker(\Id\times F^*)$ is the image of some class in 
$H^*(\cs\times\ws,\cs\times\ws)=0$, and hence must be trivial. In other words, non-triviality of $\Ker(\Id\times F^*)$ 
implies that motion planning in $\cs\times\ws$ requires more than one robust roadmap. We are going to estimate 
the minimal number of necessary roadmaps by considering products of cohomology classes. 

Let $u_1,\ldots,u_n$ be elements of the kernel of $\Id\times F^*$, and let $Q_1,\ldots,Q_n$ be query sets that admit 
robust roadmaps and cover the entire query set $\cs\times\ws$. By the above argument one can choose representatives 
$\bar u_k\in H^*(\cs\times\ws,\qs_k)$ such that $u_k=j^*(\bar u_k)$. Then the cohomology product 
$$u_1\cdot\ldots\cdot u_n=j^*(\bar u_1)\cdot\ldots\cdot j^*(\bar u_n)=j^*(\bar u_1\cdot\ldots\cdot \bar u_n)=0,$$
because $\bar u_1\cdot\ldots\cdot \bar u_n\in H^*(\cs\times\ws,\qs_1\cup\ldots\qs_n)=H^*(\cs\times\ws,\cs\times\ws)=0$. 
Therefore, if $\cs\times\ws$ can be covered by $n$ robust roadmaps, then every 
product of $n$ elements of the kernel of $\Id\times F^*$ must be 0. The minimal $n$ for which the product of any $n$ 
elements in the ideal $\Ker(\Id\times F^*)$ is zero is called the \emph{nilpotency} of the ideal, and is denoted
 $\nil(\Ker(\Id\times F^*))$. We may now state the main result of this section.

\begin{theorem}
\label{thm: nilpotency}
The complexity of the map $F\colon\cs\to\ws$ is bounded below by $$\cx(F)\ge\nil (\Ker(\Id\times F^*)).$$
\end{theorem}

See Section \ref{subsec: 6DOF} for an application of the cohomological estimate. 

\subsection{General case}

Let us now consider the general case that arises in practice, i.e. a forward kinematic map $F\colon\cs\to\ws$ 
that may have some singularities and that admit only partially defined inverse kinematic maps. 
The results and methods obtained in the previous sections will produce estimates of the complexity of $F$ 
over subspaces of the query space. In order to combine these into global estimates for the complexity of $F$, 
we will need a concept of relative complexity of $F$ over  subspaces of the query space.

Let $\as$ be any ENR subspace of the query space $\cs\times\ws$. The \emph{relative complexity} $\cx(F|\as)$ of $F$ over $\as$ is the minimal number of robust roadmaps 
that are needed to cover all points of $\as$. 

The relative complexity subsume as special instance several previously studied concepts. If we take the identity map on $\cs$ and a subset $\as\subseteq\cs\times\cs$
then $\cx(\Id_\cs|\as)=\cx(\cs|\as)$ coincides with the relative topological complexity (cf. \cite[Section 4.3]{Farber:ITR}). In many applications $\as$ will be a product
of the form $\cs'\times\ws'$ where $\cs'\subseteq\cs$ and $\ws'\subseteq\ws$, corresponding to the complexity of the task to navigate from configurations in $\cs'$ 
to end-effector positions in $\ws'$.  Another important
special case is $\cx(\cs|{c_0}\times \cs)$ which coincides with the \emph{Lusternik-Schnirelmann category} of $\cs$ (cf. \cite[Lemma 4.29]{Farber:ITR}). 

In the next proposition we collect the main properties of the relative complexity. 

\begin{proposition}
\label{prop: rel cx}
Let $F\colon\cs\to\ws$ be a map, whose graph we denote by ${\rm gr}(F)$, and let $\as,\as'$ be ENR subspaces of $\cs\times \ws$. Then the relative complexity of 
$F$ satisfies the following relations.
\begin{enumerate}
\item $\cx(F|{\rm gr}(F))=1$;
\item $\as\subseteq\as' \  \  \implies \  \  \cx(F|\as)\le\cx(F|\as')$;
\item $\cx(F|\as\cup\as')\le\cx(F|\as)+\cx(F|\as')$;
\item If $\as'$ can be horizontally deformed into $\as$ then $\cx(F|\as)\ge\cx(F|\as')$.
\end{enumerate}
\end{proposition}
\begin{proof}
The first three statements are self-evident, and only the last requires some proof. A horizontal deformation of $\as'$ into $\as$ is a continuous 
map $D\colon \as'\times [0,1]\to\cs\times \ws$, such that $D(c,w,0)=(c,w)$, $D(c,w,1)\in\as$ and ${\rm pr}_\ws(D(c,w,t))=w$ for all $(c,w)\in\as'$ and $t\in[0,1]$.
Let $\rho\colon \qs\to\pcs$ be a robust roadmap for some $\qs\subseteq\as$. Then we may define a robust roadmap $\rho'\colon\qs'\to\pcs$, where
$\qs'=\{(c,w)\in\as'\mid D(c,w,1)\in\qs\}$ and 
$$\rho'(c,w)(t)=\left\{\begin{array}{ll}
{\rm pr}_\cs (D(c,w,2t)) &;\ \ 0\le t\le \frac{1}{2}\\
\rho(D(c,w,1))(2t-1) &; \ \ \frac{1}{2}\le t\le 1 \end{array}\right.
$$
If the domains $\qs_1,\ldots,\qs_n$ cover $\as$, then the corresponding domains $\qs'_1,\ldots,\qs'_n$ cover $\as'$, therefore $\cx(F|\as')\le\cx(F|\as)$.
\end{proof}

Our next objective is to extend the results of the previous subsections to forward kinematic maps that have singular points. Let $\ws^r$ denote
the set of \emph{regular values} of the forward kinematic map $F\colon\cs\to\ws$, i.e. the set of $w\in\ws$ such that all configurations in the pre-image 
$F^{-1}(w)\subset \cs$ are regular for $F$. In practice this means that whenever the robot device position is in $\ws^r$ it can been moved in all directions in 
the working space, regardless of the position of joints. Moreover, let $\cs^r:=F^{-1}(\ws^r)$ the subspace of regular configurations that map to positions in $\ws^r$.
Then the restriction $F\colon\cs^r\to\ws^r$ is regular and we may extend our previous results on the complexity of regular maps.

\begin{theorem}
\label{thm: rel F fibration}
Let $F\colon\cs^r\to\ws^r$ be the restriction of the forward kinematic map to the subspace of $\cs$ where $F$ has regular values. Then
\begin{enumerate}
\item $\cx(F|\cs^r\times\ws^r)\le \cx(\ws^r)$;
\item $\cx(F|\cs\times\ws^r)\le \cat(\cs\times\ws^r)$.
\end{enumerate}
\end{theorem}
\begin{proof}
Statement (1) is a direct application of Theorem \ref{thm: F fibration}. As for the second claim, recall that an ENR subset $A\subseteq X$ is \emph{categorical} 
if it can be deformed 
to a point within $X$, and $\cat(X)$ is the minimal number of categorical subsets needed to cover $X$. Therefore, in order to prove (2) it is sufficient to show
that every categorical subset  $\as\subseteq\cs\times\ws^r$ admits a roadmap $\rho\colon\as\to\pcs$. Let $D\colon\as\times [0,1] \to\cs\times\ws^r$ be a deformation
of $\as$ to a point $(c_0,w_0)$ (and we may assume without loss of generality that $(c_0,w_0)\in{\rm gr}(F)$).
For every $(c,w)\in\as$ we define 
$$
\rho(c,w)(t):=\left\{\begin{array}{ll}
{\rm pr}_\cs(D(c,w,2t)) &; \ \ 0\le t\le\frac{1}{2}\\
\Gamma(c_0,{\rm pr}_\ws(D^-(c,w,-))(2t-1) &; \ \ \frac{1}{2} \le t \le 1
\end{array}\right.
$$
where $\Gamma$ is the path-lifting function for the regular map $F\colon\cs^r\to\ws^r$ and $D^-(c,w,-)$ is the reverse of the path $D(c,w,-)$. 
Note that $\Gamma$ may be applied because the path ${\rm pr}_\ws(D^-(c,w,-))$ is entirely contained in $\ws^r$. It is easy to verify that 
$\rho\colon \as\to\pcs$ is a robust roadmap. As every categorical subset of $\cs\times\ws^r$ admits a roadmap, we conclude that the complexity 
of $F$ over $\cs\times \ws^r$ does not exceed the category of $\cs\times\ws^r$.
\end{proof}

Note that if $F$ is regular then part (2) gives the estimate $\cx(F)\le\cat(\cs\times\ws)$, that we haven't mention previously
because it is usually weaker than the estimate that we proved in Theorem \ref{thm: F fibration}. For relative complexity the
situation is different as the two estimates refer to different sets of queries.

Inverse kinematics for kinematic maps with singularities can be quite complicated. In fact, even  if the map
$F\colon\cs\to\ws$ is regular one cannot expect to find in general a global inverse kinematic map - 
cf. \cite{Gottlieb:IEEE} for a discussion of obstructions to the existence of inverse kinematic maps. 
Thus we will usually first partition the set of regular values of $F$ into subspaces that admit robust inverse kinematic 
maps, and then study separately the possibility to extend motion plans near the singular values of $F$. 

Let us assume that the forward kinematic 
map $F\colon\cs\to\ws$ admits a continuous inverse kinematic map $I\colon\ws'\to\cs$ over some subspace $\ws'$ of the working
space of the robot. 
The next theorem gives a general estimate of the relative complexity and two important cases when equality holds.

\begin{theorem}
\label{thm: relative sectioned fibration}
Assume that the map $F\colon\cs\to\ws$ admits a continuous partial right inverse $I\colon\ws'\to\cs$ over a subset
 $\ws'\subseteq\ws$, and let $\cs':=F^{-1}(\ws')$. Then
\begin{enumerate}
\item $\cx(\ws|\ws'\times\ws')\le \cx(F|\cs\times\ws')\le \cx(\cs|\cs\times I(\ws'));$
\item If $\cs'$ can be deformed to $I(\ws')$, then $\cx(F|\cs\times\ws')=\cx(\cs|\cs\times I(\ws'));$
\item If $\ws'\subseteq\ws^r$, then $\cx(F|\cs'\times\ws')=\cx(\ws|\ws'\times\ws').$
\end{enumerate}
\end{theorem}
\begin{proof}
To verify statement (1) consider  
$$\cx(\ws|\ws'\times\ws')\le \cx(F|I(\ws')\times\ws')\le\cx(F|\cs\times\ws')\le \cx(\cs|\cs\times I(\ws')),$$
where the first and third inequalities are proved by the same argument as in Theorem \ref{thm: sectioned F}
while the second inequality follows from Proposition \ref{prop: rel cx}(2). Similarly, the statement (3) 
is analogous to Corollary \ref{cor: sectioned fibration} and is proved by combining relative versions of Theorems \ref{thm: sectioned F} 
and \ref{thm: F fibration}.

As for (2) observe that in general the requirement that a roadmap for a query $(c,w)$ ends in $I(w)$ is quite restrictive and we will give an example 
in Section \ref{subsec universal joint} where 
$\cx(\cs|\cs\times I(\ws'))$ is strictly bigger than $\cx(F|\cs\times\ws')$. However, if $\cs'$ can be deformed to $I(\ws')$, then we will show 
that  $\cx(F|\cs\times\ws')\ge \cx(\cs|\cs\times I(\ws'))$, and so the two complexities coincide. 
Let $D\colon \cs'\times [0,1]\to \cs$ be a deformation such that $D(c,0)=c$ and $D(c,1)=I(F(D(c,1)))$ for every $c\in\cs'$, and assume that
$\as\subseteq \cs\times \ws'$ admits a robust roadmap $\rho\colon\as\to\pcs$. Then 
the formula
$$
\bar\rho(c,c')(t):=\left\{\begin{array}{ll}
\rho(c,F(c'))(2t) &; \ \ 0\le t \le\frac{1}{2}\\
D\big(\rho(c,F(c'))(1),2t-1\big) &; \ \ \frac{1}{2}\le t \le 1
\end{array}\right.
$$
defines a roadmap $\bar\rho\colon\bar\as\to\pcs$ where $\bar\as=\{(c,c')\in \cs\times I(\ws')\mid (c,F(c'))\in\as\}$.  
If the domains $\as_1,\ldots,\as_n$ cover $\cs\times\ws'$, then the corresponding domains $\bar\as_1,\ldots,\bar\as_n$ cover $\cs\times I(\ws')$, therefore
$\cx(F|\cs\times\ws')\ge \cx(\cs|\cs\times I(\ws'))$. 
\end{proof}

In many applications the configurations space $\cs$ is given as the cartesian product of circles (namely, parameter spaces of individual joints), 
therefore it possesses the structure of a topological group. The following result then allows a precise computation of the upper bound in the above theorem.

\begin{theorem}
\label{thm: top group}
If $\cs$ is a topological group, then 
$$\cx(\cs|\cs\times\cs')=\cx(\cs)=\cat(\cs)$$
for every nonempty subspace $\cs'\subseteq\cs$. 
\end{theorem}
\begin{proof}
If $\cs$ is a topological group, then the complexity of $\cs$ is equal to its Lusternik-Schnirelmann category - see \cite[Lemma 8.2 ]{Farber:IRM}.
Let $c_0$ be any configuration in $\cs'$. Then $\cat(\cs)=\cx(\cs|\cs\times\{c_0\})$ by \cite[Lemma 4.29]{Farber:ITR} so we obtain the following 
chain of (in)equalities
$$\cat(\cs)=\cx(\cs|\cs\times\{c_0\})\le\cx(\cs|\cs\times\cs')\le\cx(\cs)=\cat(\cs).$$
\end{proof}

In our final result we will relate the complexity of $F$ to the number of partial right inverses (also called partial \emph{sections}) 
that are needed to cover all points in its codomain. 
Let $\rho\colon\qs\to \pcs$ be a partial robust roadmap over some $\qs\subseteq \cs\times\ws$, and let $\qs':=\{w\in\ws\mid (c_0,w)\in\qs\}$ for some 
fixed element $c_0\in\cs$. Then the formula $I_\rho(w):=\rho(c_0,w)(1)$ defines a robust map $I_\rho\colon \qs'\to\cs$ such that $F(I_\rho(w))=w$ 
for all $w\in\qs'$, therefore $I_\rho$ is a partial section of $F$. Moreover, one can define a deformation 
$D_\rho\colon I_\rho(\qs')\times [0,1]\to\cs$ as $D_\rho(c,t):=\rho(c_0,F(c))(t)$, so that $D_\rho(c,0)=c_0$ and $D_\rho(c,1)=c$ for all $c\in I_\rho(\qs')$.
In other words, $I_\rho$ has the additional property that its image can be deformed within $\cs$ to a point. We will say that $I_\rho$ is a \emph{categorical
partial section} for $F$. Let $\gsec(F)$ denote the minimal number of partial sections of $F$ that are needed to cover all points of $\ws$, and let
$\csec(F)$ be the minimal number of categorical partial sections of $F$ that are needed to cover all points of $\ws$. Clearly $\gsec(F)\le\csec(F)$. 

\begin{theorem}
The complexity of a map $F\colon\cs\to\ws$ satisfies the following inequality
$$\csec(F)\le\cx(F)\le\sum_{k=1}^{\gsec(F)} \cx(\cs|\cs\times I_k(\qs_k)),$$
where $I_k\colon\qs_k\to\cs$ are partial sections of $F$ and $\ws=\qs_1\cup\ldots\cup\qs_{\gsec(F)}$.
In particular, if $\cx(\cs|\cs\times I_k(\qs_k))=1$ for all $k$, then $\cx(F)=\gsec(F)=\csec(F)$.
\end{theorem}
\begin{proof}
Let $\rho_1,\ldots,\rho_{\cx(F)}$ be some minimal set of roadmaps that cover $\cs\times\ws$. By the above discussion, there exist
categorical partial sections $I_{\rho_1},\ldots,I_{\rho_{\csec(F)}}$  whose domains cover $\ws$, therefore $\csec(F)\le\cx(F)$. 

The second inequality follows from Proposition \ref{prop: rel cx}(3) and Theorem \ref{thm: relative sectioned fibration}(1) as we have
$$\cx(F)\le\sum_{k=1}^{\gsec(F)} \cx(F|\cs\times \qs_k)\le \sum_{k=1}^{\gsec(F)} \cx(\cs|\cs\times I_k(\qs_k)).$$
\end{proof}

Note that in the last theorem we did not assume that $F$ is regular at any point, so the inequality may be applied to estimate 
the complexity of arbitrary continuous maps.

\section{Examples and computations}

In this section we consider several examples of robot mechanisms that arise in practice and apply our results to estimate the motion planning complexity of 
their forward kinematic maps.

\subsection{One revolute joint}

Let us begin with a system consisting of one revolute joint as in Fig. \ref{fig: one joint}. 
Its configuration space $\cs$ is the set of all possible angles of rotation 
of the joint, and so it may be identified with the unit circle $T$. If the end-effector is the tip of the arm, then the working space $\ws$ is 
also the circle and the forward kinematic map is the identity. The complexity $\cx(\Id_\cs)=\cx(T)$ which is known to be equal 2 (see \cite[p.213]{Farber:TCMP}
for explicit description of roadmaps). 
\begin{figure}[ht]
    \centering
    \includegraphics[scale=0.5]{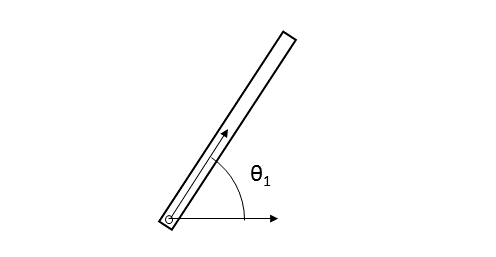}
    \caption{Position of the arm is completely described by the displacement angle $\theta$. Configuration space is $T$.}
    \label{fig: one joint}
\end{figure}
If instead of the identity map we imagine the end-effector to be connected to the joint by some transmission mechanism
so that one full rotation of the joint correspond to several rotations of the end-effector then the configuration space and working space are still equal to $T$ 
but the kinematic map  may be described as $F(\theta)=k\theta$ where $k$ is the transmission ratio. It is clear that the kinematic map is regular but, somewhat
surprisingly, it does not admit an inverse kinematic map (unless $k=1$). In fact, to define the inverse kinematic map $I\colon T\to T$ one must first define 
the value of $I(0)=\theta_0$ among angles $\theta$ for which $k\theta$ is a multiple of $2\pi$. This choice then uniquely defines $I(\theta)=\theta_0+\theta/k$
for $\theta\in [0,2\pi)$, but if $k\ne 1$ the resulting map is clearly not continuous when $\theta$ approaches $2\pi$. By Theorems \ref{thm: cx 1 implies section}
and \ref{thm: F fibration} we have
$$1<\cx(F)\le\cx(T)=2,$$
and so $\cx(F)=2$. Observe that the proof of Theorem \ref{thm: F fibration} also provides explicit roadmaps for $F$.

It often happens in practice that the configuration space or the working space are constrained for various reasons. Imagine for example that the robot arm
can rotate only some finite amount of full circles (e.g. because of the wiring). Then the configuration space is just an interval, say of the form 
$\cs=[-\theta_{max},\theta_{max}]$. We may interpret this situation as an instance of relative complexity $\cx(F|[-\theta_{max},\theta_{max}]\times\ws)$ which is 
as before bounded above by $\cx(T)=2$, but it cannot be equal to 1 because $F$ does not admit a section. 

Finally, assume that both the configuration and working space are restricted to intervals, and the forward kinematic map is a bicontinuous bijection.
Then the complexity of $F$ equals the complexity of the interval, which is 1 by \cite[Theorem 1]{Farber:TCMP}. However, if the end-effector is connected to a transmission
mechanism so that it takes several full rotations of the joint to move the end-effector between two position in the working space, then
there does not exist a continuous inverse kinematic map and one still needs two robust roadmaps to navigate this simple mechanism.

\subsection{Two planar revolute joints}

Given two revolute joints, they may be pinned together so that they rotate in the same plane or in different (usually perpendicular) planes. 
We begin with the planar case, which is simpler - see Fig. \ref{fig: two planar joints}.  If the rotation around both joints is unobstructed 
then the configuration space may be identified 
with the cartesian product of two circles, $\cs=T^2$. If, as is normally the case, the first arm is longer than the second, then the working space
of the end-effector is easily seen to be an annulus $\ws=\{(x,y)\in\RR^2\mid R_1-R_2\le \sqrt{x^2+y^2}\le R_1+R_2\}$. 
The forward kinematic map $F\colon\cs\to\ws$ can be described using polar coordinates as 
$$F(\theta_1,\theta_2)=\big(R_1\cos\theta_1+R_2\cos(\theta_1+\theta_2),R_1\sin\theta_1+R_2\sin(\theta_1+\theta_2)\big).$$
To each position in the working space correspond two joint configurations, except when both arms are parallel, so it is easy to explicitly determine 
an inverse kinematic function $I\colon\ws\to\cs$, e.g. by always choosing the `elbow down' joint position. 
\begin{figure}[ht]
    \centering
    \includegraphics[scale=0.5]{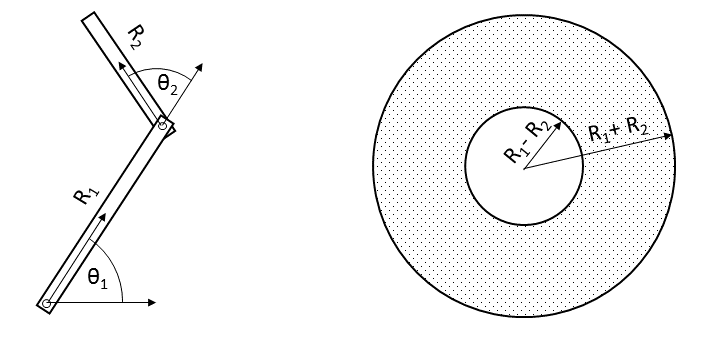}
    \caption{Position of the arm is completely described by the angles $\theta_1$ and $\theta_2$. Configuration space is $T^2$, while the 
    working space is the annulus $T\times[R_1-R_2,R_1+R_2]$.}
    \label{fig: two planar joints}
\end{figure}

Thus, by Theorem \ref{thm: sectioned F}, together
with the computations of complexity for $T$ and $T^2$ (see \cite[Theorem 12]{Farber:TCMP}) we have 
$$2=\cx(T\times [R_1-R_2,R_1+R_2])\le \cx(F) \le \cx(T^2)=3.$$
To obtain the precise value, observe that the restriction of $F$ on $\cs'=\{(\theta_1,0)\mid\theta_1\in T\}\subset\cs$ is injective, so a roadmap from $\cs$ to $F(\cs')$ 
is essentially the same as a roadmap in $\cs$ from $\cs$ to $\cs'$. Then we may 
apply Proposition \ref{prop: rel cx} and Theorem \ref{thm: top group} to get
$$\cx(F)\ge\cx(F|\cs\times F(\cs'))=\cx(\cs|\cs\times\cs')=\cx(\cs)=3,$$
therefore the complexity of $F$ is 3. Explicit roadmaps for $F$ can be derived from the proof of Theorem \ref{thm: sectioned F}.

We may extend the above reasoning to a system of $n$ planar joints. In fact, the configuration space is the cartesian product of $n$ circles, $\cs=T^n$,
while the working space $\ws$ is either a disk or an annulus, depending on the relative lengths of the robot arms. The forward kinematic map $F\colon\cs\to\ws$
is given by
\begin{align*}
F(\theta_1,\ldots,\theta_n)=&\big(R_1\cos\theta_1+R_2\cos(\theta_1+\theta_2)+\ldots+R_n\cos(\theta_1+\ldots+\theta_n),\\
 & \ \, R_1\sin\theta_1+R_2\sin(\theta_1+\theta_2)+\ldots+R_n\sin(\theta_1+\ldots+\theta_n)\big),
\end{align*} 
and it is easy to see that it admits an inverse kinematic map $I\colon\ws\to\cs$. 
For example, if the length of the first arm exceeds the sum of the lengths of the remaining arms (so that $\ws$ is an annulus),
then one may define inverse kinematics by letting $\theta_3=\ldots=\theta_n=0$ and choosing $\theta_1$ and $\theta_2$ as in 
the two-arm case. Furthermore, the restriction of $F$ on 
$\cs'=\{(\theta_1,0,\ldots,0)\mid \theta_1\in T\}\subset\cs$ is injective, so we have 
$$\cx(\cs)=\cx(\cs|\cs\times\cs')= \cx(F|\cs\times F(\cs'))\le\cx(F)\le\cx(\cs).$$
Therefore, by \cite[Theorem 12]{Farber:TCMP}, $\cx(F)=\cx(T^n)=n+1$.

\subsection{Universal joint}
\label{subsec universal joint}
Universal joint (also called Cardan joint) consists of two revolute joints whose axes intersect orthogonally, as in Fig. \ref{fig: universal joint}. 
The joint has two degrees of freedom, 
its configurations space is the cartesian product of two circles, $\cs=T^2$, and the space of positions that can be reached by its end-effector
can be parametrized by the points on the two-dimensional sphere, so $\ws=S^2$.
\begin{figure}[ht]
    \centering
    \includegraphics[scale=0.5]{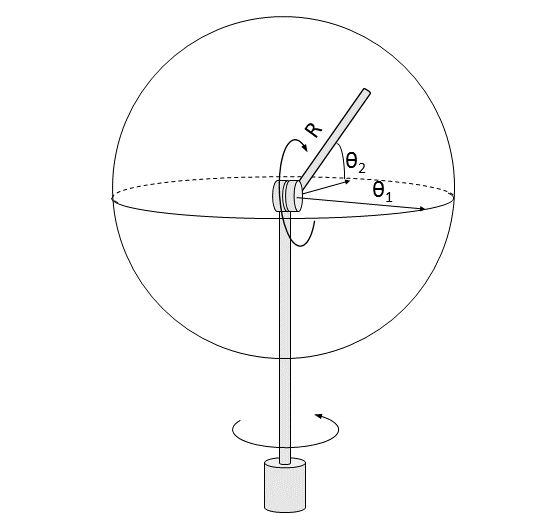}
    \caption{Position of the arm is completely described by the angles $\theta_1$ and $\theta_2$. The end-effector can reach every point on the sphere, 
    centered at the second joint. Configuration space is $T^2$ and the working space is the sphere $\ws=S^2$.}
    \label{fig: universal joint}
\end{figure}
The forward kinematic map can be described using spherical coordinates
$$F(\theta_1,\theta_2)=(R\cos\theta_1\cos\theta_2,R\sin\theta_1\cos\theta_2,R\sin\theta_2).$$
The computation of the Jacobian matrix detects the two well-known gimbal lock positions, namely when $\theta_2=\pm\frac{\pi}{2}$ and the end-effector points
at the north ($N$) or south ($S$) pole of the sphere. Therefore $\ws^r=S^2-\{N,S\}$, and $\cs^r=T\times \big(T-\{-\frac{\pi}{2},\frac{\pi}{2}\}\big)$.
The map $I\colon\ws^r\to\cs^r$ that to each point assigns its unique longitude $\theta_1$ and latitude $\theta_2$ satisfying the requirement that 
$\theta_2\in (-\frac{\pi}{2},\frac{\pi}{2})$, is clearly a robust inverse kinematic map for $F$. In order to compute the complexity of $F$ we begin with the 
following estimate based on properties (2) and (3) of Proposition \ref{prop: rel cx}:
$$\cx(F|\cs\times \ws^r)\le\cx(F)\le \cx(F|\cs\times \ws^r)+\cx(F|\cs\times \{N,S\}).$$
Note that we can define another inverse kinematic map $I'\colon\ws^r\to\cs^r$ by choosing latitude in the interval $(\frac{\pi}{2},\frac{3\pi}{2})$, and 
that $\cs^r=I(\ws^r)\cup I'(\ws^r)$. It is easy to see that $I'(\ws^r)$ can be horizontally deformed within $\cs$ to $I(\ws^r)$.
Therefore $\cs^r$ can be deformed into $I(\ws^r)$, so we may apply 
Theorem \ref{thm: relative sectioned fibration}(2), and then Theorem \ref{thm: top group} and \cite[Theorem 12]{Farber:TCMP} to compute 
$$\cx(F|\cs\times \ws^r)=\cx(\cs|\cs\times I(\ws^r))=\cx(\cs)=3.$$
To determine $\cx(F|\cs\times \{N,S\})$ we first assume that $\cx(F|\cs\times \{N\})=1$. Then, by the proof of Proposition \ref{prop: cx=1} and the comments after it,
there exists a horizontal deformation of $\cs\times\{N\}$ to the graph of $F$. But a horizontal deformation would contract $\cs$ to a point within
$\cs\times\ws$, which is clearly impossible, therefore $\cx(F|\cs\times \{N\})>1$. On the other side, we may define two explicit roadmaps over $\cs\times\{N\}$: let 
$$\qs_1:=\left\{(\theta_1,\theta_2,N)\in \cs\times\{N\}\mid \theta_2\neq -\frac{\pi}{2}\right\},$$ 
$$\rho_1(\theta_1,\theta_2,N)(t):=\left(\theta_1,(1-t)\theta_2+\frac{t\pi}{2}\right),$$
and
$$\qs_2:=\left\{(\theta_1,\theta_2,N)\in \cs\times\{N\}\mid \theta_2= -\frac{\pi}{2}\right\}, $$
$$\rho_2(\theta_1,-\frac{\pi}{2},N)(t):=\left(\theta_1,-\frac{\pi}{2}+t\pi\right).$$
Analogous formulas define  roadmaps $\rho'_1\colon\qs'_1\to\pcs$ and $\rho'_2\colon\qs'_2\to\pcs$ for $\cs\times\{S\}$. 
Since $\cs\times\{N\}$ and $\cs\times\{S\}$ are disjoint, we may combine $\rho_1$ and $\rho'_1$ into a robust roadmap on $\qs_1\cup\qs'_1$, and similarly
$\rho_2$ and $\rho'_2$ into a robust roadmap on $\qs_2\cup\qs'_2$, 
which implies that $\cx(F|\cs\times\{N,S\})=2$. Note that the map $F$ admits an obvious inverse kinematics over the one point space $\{N\}$ (namely, choose
any point with $\theta_2=\frac{\pi}{2}$), but $\cx(\cs|\cs\times\{I(N)\})=\cx(\cs)=3$, so this gives an example where $\cx(\cs|\cs\times I(\ws'))$ is strictly
bigger than $\cx(F|\cs\times\ws')$.

At this point we know that the complexity of $F$ is between 3 and 5. We are going to examine the instability of the roadmaps around the singular points of $F$ 
and show that $\cx(F)$ is in fact at least 4. Assume that there exists a motion plan for $F$ that consists of robust roadmaps $\rho_i\colon\qs_i\to\pcs$ for $i=1,2,3$. 
By restriction we obtain 3 roadmaps on $\cs\times\ws^r$, which is by above computation also the minimal number of roadmaps 
necessary to cover $\cs\times\ws^r$. Therefore, we may find robust roadmaps $\bar\rho_i\colon\overline\qs_i\to\pcs$ that cover 
$\cs\times I(\ws^r)$ and for which $\rho_i(c,w)=\bar\rho_i(c,I(w))$ for every $i$ and every $(c,w)\in\qs_i\cap(\cs\times\ws^r)$. 
Moreover, if we choose a small ball $B$ around $N\in\ws$, the complexity $\cx(\cs|\cs\times I(B-\{N\}))$ is still equal to 3, so 
$\overline\qs_i\cap (\cs\times I(B-\{N\}))\ne\emptyset$ for $i=1,2,3$. It follows that for each $i=1,2,3$ we may find sequences $(c_j,w_j)$ and $(c'_j,w'_j)$
in $\qs_i$ converging to $(c,N)$, but such that $(c_j,I(w_j))$ and $(c'_j,I(w'_j))$ converge to different points 
in $\cs\times F^{-1}(N)$
(i.e. $(w_j)$ and $(w'_j)$ converge to $N$ from different directions). 
But then we would have 
$$\lim_j \rho_i(c_j,w_j)=\lim_j\bar\rho_i(c_j,I(w_j))\ne\lim_j\bar\rho_i(c'_j,I(w'_j))=\lim_j \rho_i(c'_j,w'_j),$$
which is a contradiction. We conclude that the motion planning for the universal joint requires at least four robust roadmaps.

\subsection{Triple-roll wrist} 

Triple-roll wrist is a compound joint consisting of three revolute joints whose axes pass through a common point - see Fig. \ref{fig: triple joint}. 
By rotating the individual joints the end-effector can assume any orientation in the three-dimensional space. The device has three degrees of freedom, its configuration space 
is the cartesian product of three circles, $\cs=T^3$, and the working space consists of all possible orientations of the end-effector, $\ws=SO(3)$.
There exist several ways to relate the positions of the joints to the resulting orientation, the most common being through 
the Euler angles, see \cite[Chapter 1.2]{Handbook} for explicit formulas for the forward kinematic map $F\colon\cs\to\ws$. 
\begin{figure}[ht]
    \centering
    \includegraphics[scale=0.5]{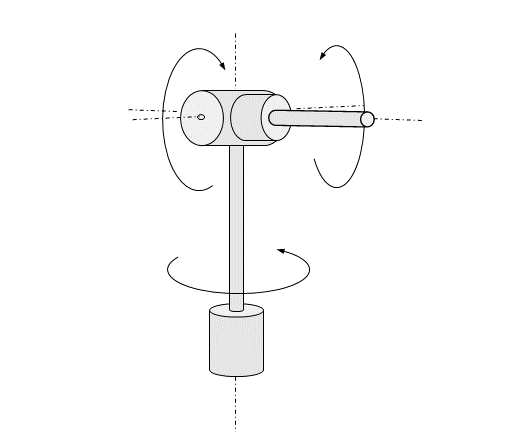}
    \caption{Position of the arm is completely described by three angles of rotation around respective joints. The end-effector can assume arbitrary orientation
    in the space. Configuration space is $T^3$ and the working space is $SO(3)$.}
    \label{fig: triple joint}
\end{figure}
The computation of $\cx(F)$ is similar to the one for the universal joint, just more complicated. We are going to sketch the main steps and leave the detailed verification
to the reader. Let $\cs=T^3$, the cartesian product of three circles, $\ws=SO(3)$ the set of orthogonal matrices with determinant 1, 
and the forward kinematic map $F\colon\cs\to\ws$ given by Euler angles, following the $X-Z-X$ convention. Then
\begin{enumerate}
\item The set of singular values $\ws^s$ of $F$ consists of all matrices of the form
$$R=\left[\begin{array}{ccc}
\cos\theta & -\sin\theta & 0 \\
\sin\theta & \cos\theta & 0\\
0 & 0 & 1 \end{array}\right]
$$
They correspond to rotations whose $Z$-axis coincide with the $Z$-axis of the reference frame. Clearly, the set of singular values $\ws^s$ may be viewed 
as a circle embedded in $SO(3)$, and $\ws^r$ is its complement in $SO(3)$. 
\item 
The singular points of $F$ are Euler triples of the form $(\theta_1,0,\theta_3)$ or $(\theta_1,\pi,\theta_3)$, because under $X-Z-X$ convention the second angle
corresponds to the rotation that moves the $Z$-axis of the reference frame to that of the represented rotation. We may therefore identify the set of singular points
with the cartesian product $T\times\{0,\pi\}\times T\subset T^3=\cs$. Geometrically speaking, that is a disjoint union ot two two-dimensional tori in $\cs$. 
\item 
To every regular value of $F$ corresponds a unique Euler triple $(\theta_1,\theta_2,\theta_3)$  where $\theta_1,\theta_3\in T$ and $\theta_2\in (0,\pi)$. 
This correspondence determines an inverse kinematic map $I\colon\ws^r\to\cs$. Alternatively, if we choose $\theta_2\in(\pi,2\pi)$ then we get another
inverse kinematic map $I'\colon\ws^r\to\cs$, and $\cs^r=I(\ws^r)\cup I'(\ws^r)$. As in the universal joint case, $\cs^r$ may be deformed within $\cs$ to $I(\ws^r)$.
\end{enumerate}

We may now proceed to the computation of $\cx(F)$: by Proposition \ref{prop: rel cx}
$$\cx(F|\cs\times\ws^r) \le \cx(F) \le \cx(F|\cs\times\ws^r)+\cx(F|\cs\times\ws^s).$$
Furthermore, by Theorems \ref{thm: relative sectioned fibration}(2), \ref{thm: top group} and \cite[Theorem 12]{Farber:TCMP}
$$\cx(F|\cs\times\ws^r)=\cx(\cs|\cs\times I(\ws^r))=\cs(T^3)=4.$$ 
Finally, one may use a similar approach as in the previous subsection to
construct two roadmaps from $T^3$ to $T^2\times \{0,\pi\}$ and show that $\cx(F|\cs\times\ws^s)=2$. Thus we may conclude that the complexity
of the triple-roll wrist is between 4 and 6 (and we believe that by a closer analysis of the singular points of $F$ one may actually prove that the complexity
of $F$ is at least 5). 

\subsection{6-DOF joint}
\label{subsec: 6DOF}

Six-degree-of-freedom serial manipulators are among the most common robot arm structure used in various applications. The precise analysis of 
the corresponding forward kinematic map $F\colon T^6\to\RR^3\times SO(3)$ is difficult and depends on the exact configuration of the joints.
Nevertheless, the underlying topology of the configuration and working spaces allows an estimate of the complexity of $F$, regardless 
of the specific joint configuration. To this end we use the cohomological lower bound that was described in \ref{subsec: cohomology}. 

The following results can be found in \cite{Hatcher 02}. The cohomology with $\ZZ_2$ coefficients of the six-dimensional torus $T^6$ and of 
the space of rotations $SO(3)$ are given as follows: the cohomology ring of $T^6$
is $H^*(T^6)=\wedge(x_1,\ldots,x_6)$, the exterior $\ZZ_2$-algebra on 6 generators in dimension 1, while the cohomology ring of $SO(3)$ is
$H^*(SO(3))=\ZZ_2[u]/(u^4)$, the truncated polynomial algebra with a 1-dimensional generator and the relation $u^4=0$. 
A full rotation around some fixed axis represents a homotopically non-trivial loop in $SO(3)$ (in fact the non-trivial element of its fundamental group). 
Since a 6-DOF mechanism allows full rotations we may conclude that the induced homomorphism $F^*\colon H^*(SO(3))\to H^*(T^6)$ is non-trivial, so the image 
of the generator $F^*(u)=s\in H^1(T^6)$ is a non-trivial sum of generators $x_1,\ldots,x_6$.

Based on these facts, it is easy to check that the cohomology class $s\times 1+ 1\times u\in H^*(T^6\times SO(3))$ is contained in the 
kernel of the homomorphism $(1\times F)^*\colon H^*(T^6\times)\to H^*(T^6)$. Moreover, by taking into account that $s^2=0$ and that 
the addition is modulo 2, we obtain 
$$(s\times 1+ 1\times u)^2=1\times u^2\ \ \text{and}\ \ (s\times 1+ 1\times u)^3=s\times u^2+1\times u^3\ne 0,$$
which by Theorem \ref{thm: nilpotency} implies that the complexity of $F$ is at least 4.

\section{Conclusion}
 
The applicability of topological methods in robotics is not as surprising as it may appear at first sight, especially when one considers 
qualitative questions regarding the possibility to find suitable inverse kinematic maps or to avoid singular points of certain configurations.
In this paper we introduced a new topological measure for the complexity of motion planning in robotic systems. 
Unlike the previous approaches to the complexity as developed by \cite{Farber:TCMP} and other authors, where one studies the motion
within a single space (either the configuration or the working space of a system), we constructed a more realistic model that takes into account
the forward kinematic map. We considered queries consisting of two sets of data - the initial configuration of the joints, and the requested final position 
of the end-effector. Then we studied obstructions to the existence of robust algorithms that take those queries as input, and return  movements 
of the joints that start from a given joint configuration and end with the required position of the end-effector. 

It turned out that in most cases one needs several robust algorithms in order to cover all possible queries. 
In particular, we proved that the complexities (i.e. the minimal number of distinct robust algorithms that are needed to cover all queries) 
of some basic joint configurations like the universal joint, the triple-roll wrist and the 6-DOF joint configurations are at least 4.

We believe that this new invariant reflects and explains many of the difficulties in the computation of the inverse kinematic maps and of 
the construction of explicit motion plans, and is thus one of the factors that should be taken into account in the design of specific robot-arm configurations.

\end{document}